\documentclass[a4paper,draft]{amsproc}
\usepackage{amssymb}
\usepackage{amscd}
\theoremstyle{plain}
\newtheorem*{theoA}{Theorem A}
\newtheorem*{theoB}{Theorem B}
\newtheorem*{theoC}{Theorem C}
\newtheorem*{theoD}{Theorem D}

 \newtheorem{theo}{Theorem}[section]

 \newtheorem{lem}{Lemma}[section]
 \newtheorem{cor}{Corollary}[section]
\theoremstyle{definition}
 \newtheorem{exm}{Example}[section]
 \newtheorem{ques}{Question}[section]
 
 \newtheorem{note}{Note}[section]
\theoremstyle{remark}
 \newtheorem{rem}{Remark}[section]
 
 \newcommand{\ol}{\overline}
\newcommand{\be}{\begin{equation}}
\newcommand{\ee}{\end{equation}}
\newcommand{\beas}{\begin{eqnarray*}}
\newcommand{\eeas}{\end{eqnarray*}}
\newcommand{\bea}{\begin{eqnarray}}
\newcommand{\eea}{\end{eqnarray}}

 \numberwithin{equation}{section}
\renewcommand{\leq}{\leqslant}
\renewcommand{\geq}{\geqslant}
\renewcommand{\setminus}{\smallsetminus}
\title[Entire function sharing a finite set ...]{\LARGE A note on entire functions sharing a finite set with applications to difference equations}
\numberwithin {equation}{section}
\date{}
\author{Molla Basir Ahamed}
\address{ Department of Mathematics, Kalipada Ghosh Tarai Mahavidyalya, West Bengal, 734014, India.}
\email{basir\_math\_kgtm@yahoo.com, bsrhmd117@gmail.com}
\begin{document}
\vspace{18mm} \setcounter{page}{1} \thispagestyle{empty}
\begin{abstract}
Value distribution and uniqueness problems of difference operator of an entire function have been investigated in this article. This research shows that a finite ordered entire function $ f $ when sharing a set  $ \mathcal{S}=\{\alpha(z), \beta(z)\} $ of two entire functions $ \alpha $ and $ \beta $ with $ \max\{\rho(\alpha), \rho(\beta)\}<\rho(f) $ with its difference $ \mathcal{L}^n_c(f)=\sum_{j=0}^{n}a_jf(z+jc) $, then $ \mathcal{L}^n_c(f)\equiv f $, and more importantly certain form of the function $ f $ has been found. The results in this paper improve those given by \emph{k. Liu}, \emph{X. M. Li}, \emph{J. Qi, Y. Wang and Y. Gu} etc. Some examples have been exhibited to show the condition $ \max\{\rho(\alpha), \rho(\beta)\}<\rho(f) $ is sharp in our main result. Examples have been also exhibited to show that if $ CM $ sharing is replaced by $ IM $ sharing, then conclusion of the main results ceases to hold.
\end{abstract}
\maketitle
\section{\sc Introduction}
In this paper, a meromorphic function will always be non-constant and meromorphic in the complex plane C, unless specifically stated otherwise. In what follows, we assume that the reader is familiar with the elementary Nevanlinna theory,(see \cite{Hay & 1964,Lai & 1993,Yan & Yi & 2003}). In particular, for a meromorphic function $ f $, $ \mathcal{S}(f) $ denotes the family of all meromorphic function $ \zeta $ for which $ T(r,\zeta)=S(r,f)=o(T(r,f)) $, where $ r\rightarrow\infty $ outside of a possible set of finite logarithmic measure.For convenience, we agree that $ \mathcal{S}(f) $ includes all constant functions and $ \overline{ \mathcal{S}}(f):=\mathcal{S}(f)\cup\{\infty\}. $
\par A set $ \mathcal{S} $ is called a unique range set (URSE) for a certain class of entire functions if each inverse image of the set uniquely determines a function from the given class. Let $ \mathcal{S} $ be a finite set of some entire functions and $ f $ an entire function. Then, a set $ E_f(\mathcal{S}) $  is defined as \beas E_f(\mathcal{S})=\{(z,m)\in\mathbb{C}\times\mathbb{Z} : f(z)-a(z)=0 \text{\;with multiplicity}\; m, a\in\mathcal{S}\}. \eeas If we do not the count multiplicities, then we denote the set as $ \ol E_f(\mathcal{S}) $.Assume that $ g $ is another entire funct ion. We say that $ f $ and $ g $ share $ \mathcal{S} $ $ CM $ if $ E_f(\mathcal{S})=E_g(\mathcal{S}) $, and we say $ f $ and $ g $ share $ \mathcal{S} $ $ IM $ if $ \ol E_f(\mathcal{S})=\ol E_g(\mathcal{S}) $. Thus, a set $ \mathcal{S} $ is called $ URSE $ if $ E_f(\mathcal{S})=E_g(\mathcal{S}), $ where $ f $ and $ g $ are two entire functions, then $ f\equiv g. $\par \emph{Gross} and \emph{Yang} \cite{Gro & Yan & PJAS & 1982}, first found an example of a $ URSE $ which is $ \mathcal{S}=\{z: e^z+z=0\} $, and as this an infinite set, so it is very natural to investigate whether there exists a finite $ URSE $ or not. In $ 1995,$ \emph{Yi} \cite{Yi & NMJ & 1995} who found a $ URSE $ $ \mathcal{S}=\{z: z^n+az^m+b=0\} $, where $ n>2m+4 $ and $ a,\; b $ chosen in a such a way that the equations $ z^n+az^m+b=0 $ has no repeated roots. Since then there have been many efforts to study the problem of constructing $ URSE $ time to time (see \cite{Fra & Rei & CVTA & 1998,Fuj & NMJ & 2003,Li & Yan & KMJ & 1995}). There is another study on the $ URSE $ of entire functions, which is to seek a set $ \mathcal{S} $ for which if $ E_f(\mathcal{S})=E_{f^{\prime}}(\mathcal{S}) $, then $ f\equiv f^{\prime} $. One can verify that the form of the function will be $ f(z)=ce^{z}. $ where $ c $ is a non-zero complex number. \emph{Li} and \emph{Yang} \cite{Li & Yan & JMSJ & 1999} also deduced that if $ E_f(\mathcal{S})=E_{f^{\prime}}(\mathcal{S}) $, where $ \mathcal{S}=\{a,b\} $ with $ a+b\neq 0 $, then either $ f(z)=\mathcal{A}e^z $ or $ f(z)=\mathcal{A}e^z+a+b $, where $ \mathcal{A} $ is a non-zero complex number. Later, \emph{Fang} and \emph{Zalcman} \cite{Fan & Zal & JMMA & 2003}, using the theory of normal families, proved that there exists a set $ \mathcal{S}=\{a, b, c\} $ such that $ E_f(\mathcal{S})=E_{f^{\prime}}(\mathcal{S}) $, then $ f\equiv f^{\prime}. $ A special topic widely studied in the uniqueness theory is the case when $ f(z) $ shares value(s) or set(s) with its derivatives or differential polynomials. We recall a result of this type from the preceding literature:

\begin{theoA}\cite{Li & Yan & JMSJ & 1999}
	Let $ f $ be a non-constant entire function and $ a_1, a_2 $ be two distinct complex numbers. If $ E_f(\{a_1, a_2\})=E_{f^{\prime}}(\{a_1, a_2\}) $, then $ f $ takes one of the following conclusions:
	\begin{enumerate}
		\item $ f=f^{\prime} $;
		\item $ f+f^{\prime}=a_1+a_2 $;
		\item $ f(z)=c_1e^{cz}+c_2e^{-cz}, $ with $ a_1+a_2=0 $, where $ c_1, c_2, c $ are all non-zero complex numbers satisfying $ c^2\neq 1 $ and $ c_1c_2=\frac{1}{4}a_1^2\left(1-\frac{1}{c^2}\right). $
	\end{enumerate}
\end{theoA} Recently, with a more general setting, namely $ k $-th derivative $ f^{(k)} $ or differential monomial $ M[f] $ or a more general setting namely differential polynomial $ P[f] $, the present author, deduced that, when some power of a meromorphic function $ f $ and its differential monomial (or polynomial) sharing a set \cite{Aha & CKMS & 2018} or a small function \cite{Aha & Lik & TJM & 2019}, then even in this case, the function $ f $ also assumes a certain form. From the literature of meromorphic functions on sharing value problems, finding the class of the functions satisfying some differential or difference equations gained a valuable space. \par Throughout this paper, we denote by $ \rho(f) $ and $ \lambda(f) $, the order of $ f $ and the exponent of convergence of zeros of $ f $ respectively (see \cite{Qi & Wang & Gu & ADE & 2019}): We also need the following notation: Let $ f $ be a non-constant meromorphic function, and we define difference operators as $ \Delta_c(f)=f(z+c)-f(z) $, and for $ n\geq 2 $, $ \Delta^n_c(f)=\Delta^{n-1}_c(\Delta_c(f)) $.   \par In the past recent years, the Navanlinna characteristic of $ f(z+c) $, the value distribution theory for difference polynomials, the Nevanlinna theory for the difference operator and most importantly, the difference analogue of the lemmas on the logarithmic derivative had been established (see \cite{Ban & Aha & MS & 2019,Chi & Fen & RJM & 2008,Hal & KOr & JMMA & 2006,Hal & Kor & PLMS & 2009,Heitto & JMMA & 2009,Heitto & CVEE & 2011,Liu & JMMA & 2009,Luo & Lin & JMMA & 2011,Zha & JMMA & 2010}). For these theories, since the derivative is a difference counterpart of a function, hence there has been recent study of whether the derivative $ f^{\prime} $ of a entire (meromorphic) function $ f $ can be replaced by the difference $ \Delta_c(f)=f(z+c)-f(z) $ in the above mentioned results. Number of researches have been done with difference operator (see\cite{Aha & SUBBM & 2019,
Ban & Aha & F & 2019,Ban & Aha & JCMA & 2020,Li & Gao & ADM & 2011,Li & CMFT & 2012,Li Chen & ADE & 2012,Li & Mei & Chen & 2017,Liu & JMMA & 2009,Qi & Wang & Gu & ADE & 2019})\par In this direction, in $ 2009 $, \emph{Liu} \cite{Liu & JMMA & 2009} considered the problem of sharing set by an entire function and its difference, and obtained the following result.
\begin{theoB}\cite{Liu & JMMA & 2009}
	Suppose that $ \alpha $ is is a non-zero complex number, and $ f $ is a transcendental entire function with finite order. If $ E_f (\{-\alpha, \alpha\})=E_{\Delta_c(f)} (\{-\alpha, \alpha\})$, then $ \Delta_c(f(z))=f(z) $ for all $ z\in\mathbb{C} $.
\end{theoB} \par Since $ \Delta_c(f) $ is a very special form of the setting $ \mathcal{L}_c(f):=a_1f(z+c)+a_0f(z) $, where $ a_1(\neq 0), a_0\in\mathbb{C} $ (see \; \cite{Aha & RM & 2019}). Therefore a natural quarry would be as the following:
\begin{ques}\label{qn1.1}
	Does \emph{Theorem B}\; still hold if we replace $ \Delta_c(f) $ by $ \mathcal{L}_c(f) $ ?
\end{ques}
\par From the following two examples, one can ensure that answer of \emph{Question \ref{qn1.1} } is not affirmative.
\begin{exm}
	Let $ f(z)=e^z $, and for $ c=\log(\log 2),$ suppose that \beas \mathcal{L}_c(f)=\frac{\sqrt{5}+\i\sqrt{7}}{\log 2}f(z+c)+(1-\sqrt{5}-i\sqrt{7})f(z).\eeas  It is clearly $ f $ is a finite order entire function and $ E_f(\mathcal{S})= E_{\mathcal{L}_c(f)}(\mathcal{S}) $, where $ \mathcal{S}=\{-\alpha,\alpha\} $, $ \alpha\in\mathbb{C}\setminus\{0\} $ but $ \mathcal{L}_c(f)\neq f. $
\end{exm}
\begin{exm}
	Let $ f(z)=sin(z) $ or $ cos(z) $ and \beas \mathcal{L}_{\pi}(f)=(1+i\sqrt{2})f(z+\pi)+i\sqrt{2}f(z).\eeas  Clearly $ f $ is of finite order and $ E_f(\mathcal{S})= E_{\mathcal{L}_{\pi}(f)}(\mathcal{S}) $, where $ \mathcal{S}=\{-\alpha,\alpha\} $, $ \alpha\in\mathbb{C}\setminus\{0\} $ but $ \mathcal{L}_{\pi}(f)\neq f. $
\end{exm}
\par To investigate with $ \mathcal{L}_c(f) $ to get the similar conclusion as of \emph{Theorem B}, we require some extra conditions. Instead of looking for some general setting of the difference, it is therefore reasonable to concentrate for the generalization of the shared set $ \{-\alpha,\;\alpha\} $.\par
In this direction, we recall here a question proposed by \emph{Liu} in \cite{Liu & JMMA & 2009} as follows.
\begin{ques}\label{q1.1}
	Let $ \alpha $ and $ \beta $ be two small functions of $ f $ with period $ c $. When a transcendental entire function $ f $ of finite order and its difference $ \Delta_c(f) $ share the set $ \{\alpha, \beta\} $ $ CM $, what can we say about the relationship between $ f $ and $ \Delta_c(f) $ ?
\end{ques}
   \par In connection with the \emph{Question \ref{q1.1}}, in $ 2012 $, \emph{Li} \cite{Li & CMFT & 2012} established the following result.
   \begin{theoC}\cite{Li & CMFT & 2012}
   	Suppose that $ \alpha $, $ \beta $ are two distinct entire functions, and $ f $ is a non-constant entire function with $ \rho(f)\neq 1 $ and $ \lambda(f)<\rho(f)<\infty $ such that $ \max\{\rho(\alpha), \rho(\beta)\}<\rho(f) $. If $ E_f (\{\alpha, \beta\})=E_{\Delta_c(f)} (\{\alpha, \beta\})$, then $ \Delta_c(f(z))\equiv f(z) $ for all $ z\in\mathbb{C} $.
   \end{theoC}
\begin{rem}
	Next example confirms that conclusion of \emph{Theorem C} still holds if we remove the condition $ \rho(f)\neq 1 $.
\end{rem}
\begin{exm}
	Let $ f(z)=2^{z/{2\pi}}\cos(z) $, and $ \mathcal{S}=\{\alpha(z), \beta(z)\} $, where $ \alpha(z) $ and $ \beta(z) $ are two non-constant polynomials in $ z $. It is clear that $ \rho(f)=1 $, and $ \max\{\rho(\alpha), \rho(\beta)\}=0<1=\rho(f) $, and $ E_f (\{\alpha, \beta\})=E_{\Delta_{2\pi}(f)} (\{\alpha, \beta\})$, and $ \Delta_c(f(z))=f(z) $ for all $ z\in\mathbb{C} $.
\end{exm}\par Thus, one natural question is : Can we prove Theorem C by omitting the restriction $ \rho(f)\neq 1 $. Recently, \emph{Qi, Wang} and \emph{Gu} \cite{Qi & Wang & Gu & ADE & 2019} answered this question by proving the following result.

\begin{theoD}\cite{Qi & Wang & Gu & ADE & 2019}
	   	Suppose that $ \alpha $, $ \beta $ are two distinct entire functions, and $ f $ is a non-constant entire function with $ \lambda(f)<\rho(f)<\infty $ such that $ \max\{\rho(\alpha), \rho(\beta)\}<\rho(f) $. If $ E_f (\{\alpha, \beta\})=E_{\Delta_c(f)} (\{\alpha, \beta\})$, then  $ f(z)=\mathcal{A}e^{\lambda z}$, where $ \mathcal{A} $, $ \lambda $ are two non-zero complex numbers satisfying $ e^{\lambda c}=2 $. Furthermore $ \Delta_c(f(z))\equiv f(z) $ for all $ z\in\mathbb{C} $.
\end{theoD}
\begin{note}
	In view of \emph{Examples 1.1} and \emph{1.2}, we have seen that in \emph{Theorem B}, it is not possible to replace $ \Delta_c(f) $ by $ \mathcal{L}_c(f) $ in general, but, the next two examples show that in case of \emph{Theorem D}, one can do it.
\end{note}
\begin{exm}
	Let $ f(z)=\displaystyle\left(\frac{\pi}{2}\right)^{z/c} $ and  $ \mathcal{L}_c(f)=2f(z+c)+(1-\pi)f(z). $ Let $ \mathcal{S}=\{\alpha, \beta\} $, where $ \alpha $ and $ \beta $ are two entire functions with $ \rho(\alpha)<1 $ and $ \rho(\beta)<1 $. Evidently, $ E_f(\mathcal{S})=E_{\mathcal{L}_c(f)}(\mathcal{S}) $ with $ \max\{\rho(\alpha), \rho(\beta)\}<\rho(f) $, and also $ \mathcal{L}_c(f(z))=f(z) $.
\end{exm}
\begin{exm}
	Let $ f(z)=(1+i)^{z/c} $ and \beas \mathcal{L}_c(f)=\left(1+i\sqrt{3}\right)f(z+c)+\left(\sqrt{3}-i\left(\sqrt{3}+1\right)\right)f(z). \eeas Let $ \mathcal{S}=\{\alpha, \beta\} $, where $ \alpha $ and $ \beta $ are two entire functions with $ \rho(\alpha)<1 $ and $ \rho(\beta)<1 $. Then, clearly $ E_f(\mathcal{S})=E_{\mathcal{L}_c}(\mathcal{S}) $ with $ \max\{\rho(\alpha), \rho(\beta)\}<\rho(f) $, and also $ \mathcal{L}_c(f(z))=f(z) $.
\end{exm}
\begin{note}
	It is not hard to check that the functional form is $ f(z)=\displaystyle \left(\frac{1-a_0}{a_2}\right)^{z/c}g(z) $, where $ g $ is a $ c $-periodic function, when $ f $ satisfies the relation $ \mathcal{L}_c(f)=f $.
\end{note}
\par In this paper, we are mainly concerned for a more generalization of $ \Delta_c(f) $, and $ \mathcal{L}_c(f) $, hence we define  \beas \mathcal{L}^n_c(f)=a_nf(z+nc)+a_{n-1}f(z+(n-1)c)+\ldots+a_1f(z+c)+a_0f(z),\eeas  where $ a_n(\neq 0), a_i\in\mathbb{C} $ for $ (i=0, 1, 2, \ldots, n-1) $. If we choose the coefficients as $ a_j=\displaystyle(-)^{j}\binom nj $, $ j=0, 1, 2, \ldots, n $, then $ \mathcal{L}^n_c(f)=\Delta^n_c(f) $. With this generalization, our aim is to study \emph{Theorems C} and \emph{D} further. So it is interesting to ask the following questions regarding \emph{Theorems C} and \emph{D}.
\begin{ques}
	In \emph{Theorems C} and \emph{D}, what happen if we replace $ \Delta_c(f) $ by $\mathcal{L}^n_c(f) $ ?
\end{ques}
\begin{ques}
	Can we get a corresponding result like \emph{Theorems C} and \emph{D}, in which the condition $ \max\{\rho(\alpha), \rho(\beta)\}<\rho(f) $ is sharp ?
\end{ques}
\begin{ques}
	What can be say about the specific form of the function $ f $ when $ \Delta_c(f) $ is replaced by $ \mathcal{L}_c(f) $ ?
\end{ques}
\section{\sc Main result}
In this article, we dealt with the above questions and answered them all affirmatively. Following is the main result of this paper.
\begin{theo}\label{th2.1}
Suppose that $ \alpha $, $ \beta $ are two distinct entire functions, and $ f $ is a non-constant entire function with $ \lambda(f)<\rho(f)<\infty $ such that $ \max\{\rho(\alpha), \rho(\beta)\}<\rho(f) $ and $ \mathcal{L}^n_c(f)(\not\equiv 0) $. If $ E_f (\{\alpha, \beta\})=E_{\mathcal{L}^n_c(f)} (\{\alpha, \beta\})$, then  \beas f(z)=\mathcal{A}_1\lambda_1^{z/c}+\mathcal{A}_2\lambda_2^{z/c}+\ldots+\mathcal{A}_n\lambda_n^{z/c},\eeas where $ \mathcal{A}_i, \lambda_i\in\mathbb{C}\setminus\{0\} $ for $ i=1, 2, \ldots,n $,  and $ \lambda_i $ are the roots of the equation \bea a_nw^n+a_{n-1}w^{(n-1) c}+\ldots+a_1w+a_0-1=0. \eea Furthermore $ \mathcal{L}^n_c(f(z))\equiv f(z) $ for all $ z\in\mathbb{C} $.
\end{theo}

\begin{rem}
	Entire functions satisfying \emph{Theorem \ref{th2.1}} do exist, and it is shown here for the case $ n=1 $ and $ n=3 $ only, and we discussed the case $ n=2 $ later in a corollary.
\end{rem}
\begin{exm}
	Let $ f(z)=\left(1+i\right)^{z/3} $ and $ \mathcal{L}_3(f)=2f(z+3)-(1+2i)f(z) $. Let $ \mathcal{S}=\{z^2-2,\; 2z^3-z+1\}.$ Evidently, $ \max\{\rho(\alpha), \rho(\beta)\}=0<1=\rho(f) $ and $ E_f(\mathcal{S})=E_{\mathcal{L}_{3}(f)}(\mathcal{S}) $, and $ f $ has the specific form and also satisfying the relation $ \mathcal{L}_{3}(f)\equiv f. $
\end{exm}
\begin{exm}
	Let $ f(z)=\displaystyle\left(\frac{i\sqrt{3}}{2}\right)^{z/\sqrt{2}} $ and \beas\mathcal{L}_{\sqrt{2}}(f)=-2\left(\sqrt{2}+i\sqrt{3}\right)f(z+\sqrt{2})+\left(1+i\sqrt{3}\left(\sqrt{2}+i\sqrt{3}\right)\right)f(z). \eeas Let $ \mathcal{S}=\{2z^2-3z+\sqrt{2}, \sqrt{3}z^3-\sqrt{5}z^2-i\sqrt{2}\beta\} $. Clearly, $ \max\{\rho(\alpha), \rho(\beta)\}=0<1=\rho(f) $ and $ E_f(\mathcal{S})=E_{\mathcal{L}_{\sqrt{2}}(f)}(\mathcal{S}) $, and $ f $ has the specific form and also satisfying the relation $ \mathcal{L}_{\sqrt{2}}(f)\equiv f. $
\end{exm}
\begin{exm}
	Let $ f(z)=6^{z/c}+\left(1+2\omega+3\omega^2\right)^{z/c}+\left(1+2\omega^2+3\omega\right)^{z/c} $ and $ \mathcal{S}=\{\alpha, \beta\} $, where $ \alpha $ and $ \beta $ are any two polynomials in $ z $. Then, we see that $ \max\{\rho(\alpha), \rho(\beta)\}=0<1=\rho(f) $ and $ E_f(\mathcal{S})=E_{\mathcal{L}^3_{c}(f)}(\mathcal{S}) $, where \beas  \mathcal{L}^3_{c}(f)=f(z+3c)-3f(z+2c)-15f(z+c)-17f(z). \eeas Clearly $ f $ has the specific form and also satisfying the relation $ \mathcal{L}^3_{c}(f)\equiv f $.
\end{exm}
\begin{exm}
	Let $ f(z)=\mathcal{A}^{z/c}+\mathcal{B}^{z/c}+\mathcal{C}^{z/c} $ and $ \mathcal{S}=\{\alpha, \beta\} $, where $ \alpha $ and $ \beta $ are any two polynomials in $ z $. Then, we see that $ \max\{\rho(\alpha), \rho(\beta)\}=0<1=\rho(f) $ and $ E_f(\mathcal{S})=E_{\mathcal{L}^3_{c}(f)}(\mathcal{S}) $, where \beas  \mathcal{L}^3_{c}(f)=f(z+3c)-\left(\mathcal{A}+\mathcal{B}+\mathcal{C}\right)f(z+2c)+\left(\mathcal{AB}+\mathcal{BC}+\mathcal{CA}\right)f(z+c)-\left(\mathcal{ABC}-1\right)f(z). \eeas Clearly $ f $ has the specific form and also satisfying the relation $ \mathcal{L}^3_{c}(f)\equiv f $
\end{exm}

\begin{rem}
	In \emph{Theorem \ref{th2.1}}, the condition $ \max\{\rho(\alpha), \rho(\beta)\}<\rho(f) $ can not be replaced by $ \max\{\rho(\alpha), \rho(\beta)\}=\rho(f) $ i.e., the condition is sharp, which can be seen in the next examples for the case $ n=1, 2 $ and $ 3 $ only.
\end{rem}
\begin{exm}
	Let $ f(z)=e^{2z}+e^z $, and \beas  \mathcal{L}_c(f)=\sqrt{3}f(z+c)+\frac{-\sqrt{3}-i\sqrt{4\sqrt{3}-3}}{2}f(z), \eeas where $ c=\log\left(\displaystyle\frac{\sqrt{3}+i\sqrt{4\sqrt{3}-3}}{2\sqrt{3}}\right) $. Let $ \mathcal{S}=\{\alpha(z), \beta(z)\} $, where $ \alpha(z)=\mathcal{A}e^z $ and $ \beta(z)=(1-\mathcal{A})e^z $, $ \mathcal{A}\in\mathbb{C}\setminus\bigg\{0,\displaystyle\frac{1}{2}, 1\bigg\} $. We see that $ \max\{\rho(\alpha), \rho(\beta)\}=1=\rho(f) $, and $ E_f(\mathcal{S})=E_{\mathcal{L}_c(f)}(\mathcal{S}) $ but neither $ f $ has the specific form nor satisfying the relation $ \mathcal{L}_c(f)=f. $
\end{exm}
\begin{exm}
	Let $ f(z)=e^{2z}+e^z+e^{-z} $, and $  \mathcal{L}_{\pi i}(f)=-\frac{1}{2}f(z+\pi i)-\frac{1}{2}f(z). $ Let $ \mathcal{S}=\{\alpha(z), \beta(z)\} $, where $ \alpha(z)=e^z $ and $ \beta(z)=e^{-z} $. Clearly $ \max\{\rho(\alpha), \rho(\beta)\}=1=\rho(f) $, and $ E_f(\mathcal{S})=E_{\mathcal{L}_{\pi i}(f)}(\mathcal{S}) $ but note that $ f $ is neither in the form nor satisfying the relation  $ \mathcal{L}_{\pi i}(f)=f. $
\end{exm}
\begin{exm}
	Let $ f(z)=\sin z+e^z $ and \beas \mathcal{L}^2_{\pi}(f)=\left(\frac{1+2e^{\pi}}{1-e^{2\pi}}\right)f(z+2\pi)+2f(z+\pi)+\left(\frac{\left(e^{\pi}+2\right)e^{\pi}}{e^{2\pi}-1}\right)f(z). \eeas Let $ \mathcal{S}=\{\alpha, \beta\} $, where $ \alpha(z)=\mathcal{B}e^z $ and $ \beta(z)=\left(1-\mathcal{B}\right)e^z $, $ \mathcal{B}\in\mathbb{C}\setminus\bigg\{0,\displaystyle\frac{1}{2},1\bigg\} $. We check that $ E_f(\mathcal{S})=E_{\mathcal{L}^2_{\pi}(f)}(\mathcal{S}) $ and $ \max\{\rho(\alpha), \rho(\beta)\}=1=\rho(f) $ but $ f $ is neither in the specific form nor satisfying $ \mathcal{L}^2_{\pi}(f)=f $.
\end{exm}
\begin{exm}
	Let $ f(z)=\cos z+e^{-z} $ and \beas  \mathcal{L}^3_{\pi}(f)&=&\left(\frac{e^{2\pi}+e^{\pi}-1}{e^{-3\pi}+1}\right)f(z+3\pi)+e^{2\pi}f(z+\pi)-e^{\pi}f(z+\pi)\\&&+\left(\frac{e^{2\pi}+e^{\pi}-1}{e^{-3\pi}+1}-e^{2\pi}-e^{\pi}\right)f(z). \eeas Let $ \mathcal{S}=\{\alpha, \beta\} $, where $ \alpha(z)=\mathcal{C}\cos z $ and $ \beta(z)=\left(1-\mathcal{C}\right)\cos z $, $ \mathcal{C}\in\mathbb{C}\setminus\bigg\{0,\displaystyle\frac{1}{2},1\bigg\} $. We check that $ E_f(\mathcal{S})=E_{\mathcal{L}^3_{\pi}(f)}(\mathcal{S}) $ and $ \max\{\rho(\alpha), \rho(\beta)\}=1=\rho(f) $ but $ f $ is not in the specific form and also not satisfying $ \mathcal{L}^3_{\pi}(f)=f $.
\end{exm}
\begin{rem}
	The next examples show that, conclusion of \emph{Theorem \ref{th2.1}} ceases to hold if one replace the $ CM $ sharing by $ IM $ sharing. The examples have been exhibited  for $ n=1 $ and $ n=2 $ only.
\end{rem}
\begin{exm}
	Let $ f(z)=-\displaystyle\frac{1}{2}\left(e^z+a^2e^{-z}\right)$ for $ a\in\mathbb{C}\setminus\{0\} $. We choose $ c\in \mathbb{C}\setminus\{k\pi i\}, $ for $ k\in\mathbb{Z} $ be such that \beas \mathcal{L}_c(f)=\frac{2e^c}{1-e^{2c}}f(z+c)-\frac{2}{1-e^{2c}}f(z). \eeas\par Evidently $ \ol E_f(\mathcal{S})= \ol E_{\mathcal{L}_c(f)}(\mathcal{S})$ where $ \mathcal{S}=\{\alpha, \beta\}=\{-a, a\} $ and $ \max\{\rho(\alpha), \rho(\beta)\}=0<1=\rho(f) $ but $ f $ has neither the specific form as in \emph{Theorem 1.1} nor satisfies the relation $ \mathcal{L}_c(f)\equiv f $.
\end{exm}
\begin{exm}
	Let $ f(z)=\displaystyle\frac{1}{2}\left(e^z+e^{-z}\right) $ and for $ c\in\mathbb{C}\setminus\bigg\{\displaystyle\frac{k\pi i}{2} : k\in\mathbb{Z}\bigg\}, $ \beas \mathcal{L}^2_{c}(f)=\frac{2}{1-e^{4c}}f(z+2c)-2e^{-c}f(z+c)+\frac{2e^{2c}}{e^{4c}-1}f(z). \eeas Evidently $ E_f(\mathcal{S})=E_{\mathcal{L}^2_c(f)}(\mathcal{S}) $, where $ \mathcal{S}=\{-1,1\} $ with $ \max{\rho(\alpha), \rho(\beta)}=0<1=\rho(f) $ but $ f $ has neither the specific form nor satisfying $ \mathcal{L}^2_c(f)=f. $
\end{exm}
\par For $ a_2(\neq 0), a_1, a_0\in\mathbb{C} $, we define two constants $ \mathcal{D}_1 $ and $ \mathcal{D}_2 $ as follows \beas \mathcal{D}_1=\frac{-a_1+\sqrt{a_1^2-4a_2(a_0-1)}}{2a_2}\;\; \text{and}\;\;\mathcal{D}_2=\frac{-a_1-\sqrt{a_1^2-4a_2(a_0-1)}}{2a_2}.\eeas We have a corollary of the main result.
\begin{cor}\label{c2.1}
	Suppose that $ \alpha $, $ \beta $ are two distinct entire functions, and $ f $ is a non-constant entire function with $ \lambda(f)<\rho(f)<\infty $ such that $ \max\{\rho(\alpha), \rho(\beta)\}<\rho(f) $ and $ \mathcal{L}^2_c(f)(\not\equiv 0) $. If $ E_f (\{\alpha, \beta\})=E_{\mathcal{L}^2_c(f)} (\{\alpha, \beta\})$, then
	\[
	f(z)=\begin{cases}
		\displaystyle\mathcal{A}\left(-\frac{a_1}{2a_2}\right)^{z/c} \; \text{when}\;\; a_1^2+4a_2=4a_0a_2   \\
		\mathcal{A}_1\mathcal{D}_1^{z/c}+\mathcal{A}_2\mathcal{D}_2^{z/c},\;\text{otherwise}  \\
	\end{cases}
	\]
where $\mathcal{A}, \mathcal{A}_1, \mathcal{A}_2 \in\mathbb{C}\setminus\{0\}. $ Furthermore $ \mathcal{L}^2_c(f(z))\equiv f(z) $ for all $ z\in\mathbb{C} $.
\end{cor}
\begin{rem}
	From the next examples, one can observe that entire functions in support
	of \emph{Corollary \ref{c2.1}} when (i) $ a_1^2+4a_2=4a_0a_2 $ and (ii) $ a_1^2+4a_2\neq 4a_0a_2. $ Here $ \mathcal{S}=\{\alpha, \beta\} $, where $ \alpha $ and $ \beta $ are two polynomials in $ z $.
\end{rem}
\begin{exm}
	Let $ f(z)=\displaystyle\frac{2\sqrt{2}}{\sqrt{3}}\left(\frac{\sqrt{3}}{2\sqrt{2}}\right)^{z/c} $, and \beas \mathcal{L}^2_c(f)=\sqrt{2}f(z+2c)-\sqrt{3}f(z+c)+\frac{3+4\sqrt{2}}{4\sqrt{2}}f(z).\eeas  One can check that $ a_1^2+4a_2=4a_0a_2 $, and all the conditions of \emph{Corollary \ref{c2.1}} are satisfied and $ f $ has the specific form and also $ \mathcal{L}^2_c(f)=f $.
\end{exm}
\begin{exm}
	Let \beas  f(z)=\sqrt{3}\displaystyle\left(\frac{\sqrt{16+\sqrt{17}}+\sqrt{\sqrt{17}}}{4\sqrt{2}}\right)^{z/c}+\sqrt{2}\displaystyle\left(\frac{\sqrt{16+\sqrt{17}}-\sqrt{\sqrt{17}}}{4\sqrt{2}}\right)^{z/c} \eeas and \beas \mathcal{L}^2_c(f)=2\sqrt{2}f(z+2c)-\sqrt{16+\sqrt{17}}f(z+c)+(1+\sqrt{2})f(z).\eeas Verify that $ a_1^2+4a_2\neq4a_0a_2 $, and all the conditions of \emph{Corollary \ref{c2.1}} are satisfied and $ f $ is in the form satisfying $ \mathcal{L}^2_c(f)=f $.
\end{exm}
\section{\sc Some useful lemmas}
As discussed in section 1, \emph{Halburd - Korhonen} \cite{Hal & KOr & JMMA & 2006} and \emph{Chiang - Feng} \cite{Chi & Fen & RJM & 2008} inspected the value distribution theory of difference expressions, including the difference analogue of the logarithmic derivative lemma, independently, we recall here their results.
\begin{lem}\label{lem2.1}\cite{Hal & KOr & JMMA & 2006}
	Let $ f(z) $ be a non-constant meromorphic function of finite order, and $ c\in\mathbb{C}\setminus\{0\} $, $ \delta <1. $ Then \beas m\left(r,\frac{f(z+c)}{f(z)}\right)=o\left(\frac{T(r+|z|,f)}{r^{\delta}}\right), \eeas for all $ r $ outside of a possible exceptional set with finite logarithmic measure.
\end{lem} \par By \cite[Lemma 2.1]{Hal & Kor & PLMS & 2009}, we have  $T(r+|c|, f(z)) = (1+o(1))T(r, f) $ for all r outside of a set with finite logarithmic measure, when $ f $ is of finite order.
\begin{lem}\label{lem2.2}\cite{Chi & Fen & RJM & 2008}
	Let $ f $ be a meromorphic function of finite order, and $ \eta_1 $, $ \eta_2 $ be two distinct arbitrary complex numbers. Assume that $ \sigma $ is the order of $ f $, then for each $ \epsilon>0 $, we have \beas m\left(r,\frac{f(z+\eta_1)}{f(z+\eta_2)}\right)=O\left(r^{\sigma-1+\epsilon}\right).  \eeas
\end{lem}
\begin{lem}\label{lem2.3}\cite{Ber & Lan & MPCPS & 2007}
	Let $ g $ be a function transcendental and meromorphic in the plane of order less than $ 1 $. Set $ h>0 $. Then there exists an $ \epsilon $-set $ E $ such that \beas \frac{g(z+\eta)}{g(z)}\rightarrow 1, \text{when $ z\rightarrow\infty $ in $ \mathbb{C}\setminus E$} \eeas uniformly in $ \eta $ for $ |\eta|\leq h $.
\end{lem}
\begin{lem}\label{lem2.4}
	Let $ \alpha $, $ \beta $ are two distinct entire functions, and $ f $ is a non-constant entire function such that $ \max\{\rho(\alpha), \rho(\beta)\}<\rho(f) $ and $ E_f (\{\alpha,\beta\})=E_{\mathcal{L}_c(f)} (\{\alpha, \beta\})$. Let $ f(z)=\mathcal{G}(z)e^{P(z)} $, where $ \mathcal{G}(\not\equiv 0) $ is an entire function and $ P $ be a polynomial in $ z $. If \beas  \mathcal{W}(z)=\sum_{j=0}^{n}a_j\mathcal{G}(z+jc)e^{P(z+jc)-P(z)},\eeas then $ \mathcal{W}(z)\not\equiv 0 $.
\end{lem}
\begin{proof}
	By the \emph{Hadamard factorization theorem}, since $ f $ is an entire function, we can write $ f(z)=\mathcal{G}(z)e^{P(z)} $, where $ \mathcal{G}(\not\equiv 0) $ is an entire function and $ P $ be a polynomial in $ z $. Since $ E_f (\{\alpha,\beta\})=E_{\mathcal{L}^n_c(f)} (\{\alpha, \beta\}) $, so we can write \bea\label{e2.1} \frac{(\mathcal{L}^n_c(f)-\alpha)((\mathcal{L}^n_c(f)-\beta))}{(f-\alpha)(f-\beta)}=e^{\mathcal{Q}}, \eea where $ \mathcal{Q} $ is an entire function. Furthermore, it follows from (\ref{e2.1}) and the condition $ \max\{\rho(\alpha), \rho(\beta)\}<\rho(f)<\infty, $ that $ \mathcal{Q} $ is a polynomial in $ z $. Since $ f(z)=\mathcal{G}(z)e^{P(z)} $, then we can write $ \mathcal{L}^n_c(f) $ as \bea\label{e2.2} \mathcal{L}^n_c(f) &=& \sum_{j=0}^{n}a_jf(z+jc)= \bigg[\sum_{j=0}^{n}a_j\mathcal{G}(z+jc)e^{P(z+jc)-P(z)}\bigg]e^{P(z)}. \eea\par Substituting the forms of the functions $ f $ and $ \mathcal{L}^n_c(f) $ in (\ref{e2.1}), we have \bea\label{e2.4} &&  \bigg\{\bigg[\sum_{j=0}^{n}a_j\mathcal{G}(z+jc)e^{P(z+jc)-P(z)}\bigg]e^{P(z)}-\alpha(z)\bigg\}\\&\times&\bigg\{\bigg[\sum_{j=0}^{n}a_j\mathcal{G}(z+jc)e^{P(z+jc)-P(z)}\bigg]e^{P(z)}-\beta(z)\bigg\}\nonumber\\&=& \bigg\{\mathcal{H}(z)e^{P(z)}-\alpha(z)\bigg\}\bigg\{\mathcal{H}(z)e^{P(z)}-\beta(z)\bigg\}e^{\mathcal{Q}(z)}\nonumber. \eea \par On contrary, let if possible $ \mathcal{W}\equiv 0 $. This shows that \bea\label{e2.5} \sum_{j=0}^{n}a_j\mathcal{G}(z+jc)e^{P(z+jc)-P(z)}=0. \eea\par In view of (\ref{e2.2}) and (\ref{e2.5}), we see that $ \mathcal{L}^n_c(f)\equiv 0,$ which contradicts $ \mathcal{L}^n_c(f)\not\equiv 0 $. Therefore, we must have $ \mathcal{W}(z)\not \equiv 0. $
\end{proof}
\begin{lem}\label{lem2.5}
	Let  $ f $ and $ \mathcal{G} $ be two entire functions and $ P(z) $ is a polynomial in $ zj  $, and $ a_j $, $ j=0, 1, \ldots, n $ be all complex constants with $ a_n\neq 0 $, such that \beas \sum_{j=0}^{n}a_j\mathcal{G}(z+jc)e^{P(z+jc)-P(z)}=\mathcal{G}(z), \eeas then for each $ j=1, 2, \ldots, n $ and $ \epsilon>0 $, \beas m\left(r,e^{P(z+jc)-P(z)}\right)=[\mathcal{A}+o(1)]r^{\rho(f)-1+\epsilon}, \eeas for a complex number $ \mathcal{A}. $
\end{lem}
\begin{proof}
	We prove this lemma by inductive way. Let $ n=1 $.\par Then, we have \beas a_0\mathcal{G}(z)+a_1\mathcal{G}(z+c)e^{P(z+c)-P(z)}=\mathcal{G}(z) \eeas which we can rewrite as \bea\label{ee33.55} e^{P(z+c)-P(z)}=\frac{1-a_0}{a_1}\frac{\mathcal{G}(z)}{\mathcal{G}(z+c)}. \eea In view of \emph{Lemma \ref{lem2.3}}, we get from (\ref{ee33.55}) that \beas m\left(r,e^{P(z+c)-P(z)}\right)&=& m\left(r,\frac{1-a_0}{a_1}\frac{\mathcal{G}(z)}{\mathcal{G}(z+c)}\right)\\&=& m\left(r,\frac{\mathcal{G}(z)}{\mathcal{G}(z+c)}\right)+O(1)\\&=& O\left(r^{\rho(\mathcal{G})-1+\epsilon}\right)+O(1). \eeas On the other hand we have $ m\left(r,e^{P(z+c)-P(z)}\right)=[\mathcal{A}+o(1)]r^{\rho(f)-1+\epsilon}, $ where $ \mathcal{A} $ is a complex number.\par Let $ n=2 $, then we have \bea\label{ee33.66} e^{P(z+2c)-P(z)}=\frac{1-a_0}{a_2}\frac{\mathcal{G}(z)}{\mathcal{G}(z+2c)}-\frac{a_1}{a_2}\frac{\mathcal{G}(z+c)}{\mathcal{G}(z+2c)}. \eea \par By \emph{Lemma \ref{lem2.3}}, we get from (\ref{ee33.66}) \beas && m\left(r,e^{P(z+2c)-P(z)}\right)\\&=& m\left(r,\frac{\mathcal{G}(z)}{\mathcal{G}(z+2c)}\right)+m\left(r,\frac{\mathcal{G}(z+c)}{\mathcal{G}(z+2c)}\right)+m\left(r,e^{P(z+c)-P(z)}\right)\\&&+O(1)\\&=& O\left(r^{\rho(\mathcal{G})-1+\epsilon}\right)+O(1). \eeas On the other hand we have $ m\left(r,e^{P(z+2c)-P(z)}\right)=[\mathcal{A}+o(1)]r^{\rho(f)-1+\epsilon}. $ \par So, continuing in this way, one can prove that \beas m\left(r,e^{P(z+jc)-P(z)}\right)=[\mathcal{A}+o(1)]r^{\rho(f)-1+\epsilon},\; \text{for}\;\; j=3, 4, \ldots, n.  \eeas \par This completes the proof.

\end{proof}
\section{\sc Proof of Theorem}
In this section, we give the proof of the main result of this article.
\subsection{\sc Proof of Theorem \ref{th2.1}}
From the conditions of \emph{Theorem \ref{th2.1}}, since we have $ E_f (\{\alpha,\beta\})=E_{\mathcal{L}^n_c(f)} (\{\alpha, \beta\}) $, there must exists an entire function $ \mathcal{Q} $ such that  \bea\label{e3.1} \frac{(\mathcal{L}^n_c(f)-\alpha)((\mathcal{L}^n_c(f)-\beta))}{(f-\alpha)(f-\beta)}=e^{\mathcal{Q}}. \eea \par Again, it follows from (\ref{e3.1}) and the condition $ \max\{\rho(\alpha), \rho(\beta)\}<\rho(f)<\infty, $ that $ \mathcal{Q} $ is a polynomial in $ z $. By the \emph{Hadamard factorization theorem}, we may  suppose that $ f(z)=\mathcal{G}(z)e^{P(z)} $, where $ \mathcal{G}(\not\equiv 0) $ is an entire function, and $ P $ is a polynomial satisfying $ \lambda(f)=\rho(\mathcal{G})<\rho(f)=\deg(P). $
We set \beas  \mathcal{W}(z)=\sum_{j=0}^{n}a_j\mathcal{G}(z+jc)e^{P(z+jc)-P(z)}.\eeas By \emph{Lemma \ref{lem2.4}}, we must have $ \mathcal{W}(z)\not\equiv 0 $. It is not hard to check that $ \mathcal{W} $ is a small function of $ e^{P(z)} $. Rewriting (\ref{e2.4}), we have \bea \label{e3.2} e^{\mathcal{Q}(z)}=\frac{\mathcal{W}^2(z)\bigg[e^{P(z)}-\displaystyle\frac{\alpha(z)}{\mathcal{W}(z)}\bigg]\bigg[e^{P(z)}-\displaystyle\frac{\beta(z)}{\mathcal{W}(z)}\bigg]}{\mathcal{G}^2(z)\bigg[e^{P(z)}-\displaystyle\frac{\alpha(z)}{\mathcal{G}(z)}\bigg]\bigg[e^{P(z)}-\displaystyle\frac{\beta(z)}{\mathcal{G}(z)}\bigg]}.\eea By our assumption, since $ \alpha(z)\not\equiv \beta(z) $, without any loss of generality, we may suppose that $ \alpha(z)\not\equiv 0. $\par Let $ z_0 $ be a zero of $ e^{P(z)}-\displaystyle\frac{\alpha(z)}{\mathcal{H}(z)} $ such that $ \mathcal{W}(z_0)\neq 0 $, then it follows from (\ref{e3.2}) that $ z_0 $ must be a zero of $ e^{P(z)}-\displaystyle\frac{\alpha(z)}{\mathcal{W}(z)} $ or $ e^{P(z)}-\displaystyle\frac{\beta(z)}{\mathcal{W}(z)} $. We denote now two counting functions here, one is $ N_1\left(r,e^{P}\right) $ of the common zeros of $   e^{P(z)}-\displaystyle\frac{\alpha(z)}{\mathcal{G}(z)} $  and $  e^{P(z)}-\displaystyle\frac{\alpha(z)}{\mathcal{W}(z)}, $ and the second one is  $ N_2\left(r,e^{P}\right) $ be the reduced counting of those zeros of $   e^{P(z)}-\displaystyle\frac{\alpha(z)}{\mathcal{G}(z)} $  and $  e^{P(z)}-\displaystyle\frac{\beta(z)}{\mathcal{W}(z)}. $ Since $ \mathcal{G}(z) $ is a small function of $ e^{P(z)} $, then by applying \emph{First} and \emph{Second Fundamental Theorem}, we deduce that \bea\label{e3.3} T\left(r,e^{P}\right)&=& \ol N\left(r,\frac{1}{e^{P(z)}-\displaystyle\frac{\alpha(z)}{\mathcal{H}(z)}}\right)+S\left(r,e^{P(z)}\right)\\&=& N_1\left(r,e^{P(z)}\right)+N_2\left(r,e^{P(z)}\right)+S\left(r,e^{P(z)}\right).\nonumber\eea\par It is clear from (\ref{e3.3}) that  \beas \text{either}\;\;  N_1\left(r,e^{P(z)}\right)\neq S\left(r,e^{P(z)}\right)\;\; \text{or}\;\;   N_2\left(r,e^{P(z)}\right)\neq S\left(r,e^{P(z)}\right).\eeas  \par In our next discussions, we are going to appraise the following two cases.\\
\noindent{Case 1.} Suppose $  N_1\left(r,e^{P(z)}\right)\neq S\left(r,e^{P(z)}\right). $\par  Let $ z_1 $ be a common zero of $ e^P-\displaystyle\frac{\alpha(z)}{\mathcal{G}(z)} $ and $ e^P-\displaystyle\frac{\alpha(z)}{\mathcal{W}(z)}. $ Therefore, $ z_1 $ must be a zero of \beas \left(e^P-\displaystyle\frac{\alpha(z)}{\mathcal{W}(z)}\right)-\left(e^P-\displaystyle\frac{\alpha(z)}{\mathcal{G}(z)}\right)=\displaystyle\frac{\alpha(z)}{\mathcal{G}(z)}-\displaystyle\frac{\alpha(z)}{\mathcal{W}(z)}. \eeas
\noindent{Subcase 1.1.} Let is possible, $ \mathcal{G}(z)\not\equiv \mathcal{W}(z) $. Then it follows that $\displaystyle \frac{\alpha(z)}{\mathcal{G}(z)}-\frac{\alpha(z)}{\mathcal{W}(z)}\not\equiv 0. $ We next deduce that \beas\;\;\;\;\;\;\;\;\;  S\left(r,e^{P}\right)\neq N_1\left(r,e^P\right)\leq N\left(r,\frac{1}{\displaystyle \frac{\alpha(z)}{\mathcal{H}(z)}-\frac{\alpha(z)}{\mathcal{W}(z)}}\right)\leq T\left(r,\frac{\alpha(z)}{\mathcal{H}(z)}-\frac{\alpha(z)}{\mathcal{W}(z)} \right)=S(r,e^P),  \eeas which is a contradiction.
\\ \noindent{Subcase 1.2.} Suppose $ \mathcal{G}(z)\equiv \mathcal{W}(z) $. Which in turn shows that \bea\label{e3.4} \sum_{j=0}^{n}a_j\mathcal{G}(z+jc)e^{P(z+jc)-P(z)}=\mathcal{G}(z), \eea In view of \emph{Lemma \ref{lem2.5}} and (\ref{e3.4}), for each $ j=1, 2, \ldots, n $ and $ \epsilon>0 $, we have \beas m\left(r,e^{P(z+jc)-P(z)}\right)=[\mathcal{A}+o(1)]r^{\rho(f)-1+\epsilon}, \eeas for a complex number $ \mathcal{A}. $\par Let if possible $ \rho(f)>1 $.\\ Then, by $ \rho(f)>\rho(\mathcal{G}) $ and the estimates of $ m\left(r,e^{P(z+jc)-P(z)}\right) $, we can easily get a contradiction.  \par Thus, we see that $ \rho(f)\leq 1 $, and which in turn shows that $ e^{P(z+jc)-P(z)} $ is a non-zero constant, say, $ \eta_j $ for $ j=0, 1, 2, \ldots, n $. We have $ e^{P(z+c)}=\eta_j e^{P(z)} $. Therefore, it follows from (\ref{e3.4}) that, \bea\label{e3.5} \sum_{j=0}^{n}a_j\eta_j\frac{\mathcal{G}(z+jc)}{\mathcal{G}(z)}=1.\eea Since $ \rho(\mathcal{G})<\rho(f)\leq 1 $, then in view of \emph{Lemma \ref{lem2.3}}, there exists an $ \epsilon $-set $ E $, as $ z\not\in E $ and $ |z|\rightarrow \infty $, such that for each $ j=1, 2, \ldots, n $ we can obtained \beas  \frac{\mathcal{G}(z+jc)}{\mathcal{G}(z)}\rightarrow 1 \eeas which shows that $\displaystyle \sum_{j=0}^{n}a_j\eta_j=1. $ If we approach inductively, it follows from (\ref{e3.5}) that $ \mathcal{G}(z+c)=\mathcal{G}(z) $ for all $ z\in\mathbb{C} $ i.e., $ \mathcal{G} $ is a periodic entire function of period $ c $. Let if possible  $ \mathcal{G} $ is non-constant, then $ \mathcal{G} $ must be transcendental, and hence $ \rho(\mathcal{G})\geq 1 $, which contradicts $ \rho(\mathcal{G})\leq\rho(f)< 1 $.\par Therefore, $ \mathcal{G} $ is constant. Since, $ f $ is a finite order non-constant entire function with $ \deg(P)=\rho(f)\leq 1 $, so we must have $ \deg(P)=1 $, and hence $ P(z) $ will be of the form $ P(z)=d_1z+d_0 $, for some $ d_1(\neq 0), d_0\in\mathbb{C} $. Therefore, we can write the function $ f $ as $ f(z)=\mathcal{C}e^{\mu z}, $ where $ \mathcal{C} $ and $ \mu $ are two non-zero complex constants. \par By our assumption in Case 1, we see that $ f-\alpha $ and $ \mathcal{L}^n_c(f)-\alpha $ have common zeros, which are not the zeros of $ \alpha(z) $. Let $ z_3 $ be a common zero of $ f-a $ and $ \mathcal{L}^n_c(f) $ so that $ \alpha(z_3)\neq 0. $ Then, we see that $ z_3 $ must be a zero
\beas f(z)-\alpha(z)+\mathcal{L}^n_c(f)-\alpha(z)=\sum_{j=1}^{n}a_jf(z+jc)+(a_0+1)f(z)-2\alpha(z).\eeas Thus we have
\[
\begin{cases}
\mathcal{C}e^{\mu z_3}-\alpha(z_3)=0\\
\displaystyle\sum_{j=1}^{n}a_j\mathcal{C}e^{\mu z_3}e^{\mu jc}+(a_0+1)\mathcal{C}e^{\mu z_3}-2\alpha(z_3)=0\\

\end{cases}
\]
\par Thus we must have $ \displaystyle\sum_{j=1}^{n}a_je^{\mu jc}+a_0=1 $ which can written as \bea\label{ee44.66} a_n\left(e^{\mu c}-\lambda_1\right)\left(e^{\mu c}-\lambda_2\right)\ldots\left(e^{\mu c}-\lambda_n\right)=0, \eea where $ \lambda_1, \ldots, \lambda_n $ are distinct roots of the equation \beas a_nz^n+a_{n-1}z^{n-1}+\ldots+a_0=1. \eeas
From (\ref{ee44.66}) , we obtained that \beas e^{\mu c}=\lambda_i,\; \text{for some}\;  i=1, 2, \ldots, n,  \eeas  and thus, the general form of the function is \beas f(z)=\mathcal{C}_1\lambda_1^{z/c}+\mathcal{C}_2\lambda_2^{z/c}+\ldots+\mathcal{C}_n\lambda_n^{z/c}, \eeas where $ \mathcal{C}_1, \mathcal{C}_2, \ldots, \mathcal{C}_n $ are all complex constants.
\\
\noindent{Case 2.} Let $ N_2\left(r,e^{P}\right)\neq S\left(r,e^{P}\right). $\par Let $ z_4 $ be a common zero of $ \displaystyle e^{P(z)}-\frac{\alpha(z)}{\mathcal{G}(z)} $  and $ \displaystyle e^{P(z)}-\frac{\beta(z)}{\mathcal{W}(z)}. $ One can check that $ z_4 $ must be a zero of  $\displaystyle \frac{\alpha(z)}{\mathcal{H}(z)}-\frac{\beta(z)}{\mathcal{W}(z)}. $\par Let $ \displaystyle \frac{\alpha(z)}{\mathcal{G}(z)}-\frac{\beta(z)}{\mathcal{W}(z)}\not\equiv 0 $. Then we see that \beas\;\;\;\;\;\;\; S\left(r,e^{P}\right)\neq N_2\left(r,e^{P}\right)\leq N\left(r,\frac{1}{\displaystyle \frac{\alpha(z)}{\mathcal{G}(z)}-\frac{\beta(z)}{\mathcal{W}(z)}}\right)\leq T\left(r,\displaystyle \frac{\alpha(z)}{\mathcal{G}(z)}-\frac{\beta(z)}{\mathcal{W}(z)}\right)=S\left(r,e^{P}\right), \eeas which is clearly a contradiction.\par Thus we have \bea\label{e3.6} \displaystyle \frac{\alpha(z)}{\mathcal{G}(z)}\equiv\frac{\beta(z)}{\mathcal{W}(z)}. \eea
Suppose $ \beta(z)\equiv 0 $, then $ \displaystyle\frac{\alpha(z)}{\mathcal{G}(z)}\equiv 0 $ i.e., $ \alpha(z)\equiv 0 $, which is absurd.\\
Let $ z_5 $ is a zero of $\displaystyle e^P-\frac{\beta}{\mathcal{G}} $ but not a zero of $ \mathcal{W} $. Then it follows from (\ref{e3.2}) that $ z_5 $ is a zero of $ \displaystyle e^P-\frac{\alpha}{\mathcal{W}} $ or $ \displaystyle e^P-\frac{\beta}{\mathcal{G}}. $
 Here, we denote by $ N_3\left(r,e^{P}\right) $ the reduced counting function of those common zeros of $ e^{P}-\frac{\beta}{\mathcal{G}} $ and $ e^{P}-\frac{\alpha}{\mathcal{W}} $. Similarly, we denote by $ N_4\left(r,e^{P}\right) $ the reduced counting function of those common zeros of $ e^{P}-\frac{\beta}{\mathcal{G}} $ and $ e^{P}-\frac{\beta}{\mathcal{W}}. $\par Applying \emph{First} and \emph{Second Fundamental Theorem}, we have \beas T\left(r,e^{P}\right)&=& \ol N\left(r,\frac{1}{e^{P}-\frac{\beta}{\mathcal{H}}}\right)+S\left(r,e^{P}\right)\\&=& N_3\left(r,e^{P}\right)+N_4\left(r,e^{P}\right)+S\left(r,e^{P}\right), \eeas which implies that either $ N_3\left(r,e^{P}\right)\neq S\left(r,e^{P}\right) $ or $ N_4\left(r,e^{P}\right)\neq S\left(r,e^{P}\right) $.\\
 \noindent{Subcase 2.1.} If $ N_4\left(r,e^{P}\right)\neq S\left(r,e^{P}\right) $, then in the same manner as in Case 1, we get our desired result.\\
 \noindent{Subcase 2.2.} Next we assume that $ N_3\left(r,e^{P}\right)\neq S\left(r,e^{P}\right) $.\\ Then similar to the Case 2, one can simply deduce that \bea\label{e3.7} \displaystyle\frac{\beta}{\mathcal{G}}-\frac{\alpha}{\mathcal{W}}\equiv 0. \eea Combining  (\ref{e3.6}) and (\ref{e3.7}) yields that $ \alpha^2=\beta^2 $ which in turn implies that $ \alpha=-\beta $ as $ \alpha\neq \beta $.
Thus, we have from (\ref{e3.7})  that $  \mathcal{W}=\mathcal{G} $. Rest of the proof of the result next follows from the Subcase 1.2.


\end{document}